\documentclass[11pt]{amsart}
\usepackage{pxfonts}
\usepackage{enumitem}
\usepackage[T1]{fontenc}

\usepackage[pdftitle= {Quasi-Abelian },pdfauthor={MMLRPST}]{hyperref}
\usepackage{amscd, amssymb, amsmath, wasysym, yfonts}
\usepackage{graphicx}
\usepackage[all]{xy}
\usepackage{url}
\usepackage[usenames, dvipsnames]{color}
\usepackage{tcolorbox}
\usepackage{hyperref}
 \hypersetup{
  colorlinks=true,
   citecolor=purple!60!black!49!blue,
   linkcolor=black,
   urlcolor=teal}
\makeatletter
\def\l@subsection{\@tocline{2}{0pt}{2.5pc}{5pc}{}}
\makeatother
   
   \usepackage{enumitem}

\usepackage{amsrefs}
\usepackage[margin=1.50in]{geometry}
\theoremstyle{plain}
\newtheorem{thm}{Theorem}

\newtheorem{introcor}[thm]{Corollary}
\newtheorem{theorem}{Theorem}[section]
\newtheorem*{thmB}{Theorem A*}
\newtheorem{prop}[theorem]{Proposition}

\newtheorem{lm}[theorem]{Lemma}

\newtheorem{cor}[theorem]{Corollary}
\newtheorem*{conj}{Conjecture}

\theoremstyle{definition}

\newtheorem{defn}[theorem]{Definition}
\newtheorem{nota}[theorem]{Notation}
\newtheorem{rmk}[theorem]{Remark}
\newtheorem{set}[theorem]{Setting}
\newtheorem{ex}[theorem]{Example}
\newtheorem*{thm*}{Theorem}
\newtheorem*{ex*}{Example}

\newcommand\sO{{\mathcal O}}

\newcommand\OO{{\mathcal O}}
\newcommand\inv{^{-1}}
\newcommand\ol{\overline}
\newcommand\wt{\widetilde}

\newcommand{\Pico}[1]{{\rm Pic}^0(#1)}

\DeclareMathOperator{\Pic}{Pic}

\newcommand\pp{{\mathbf{P}}}

\usepackage[colorinlistoftodos]{todonotes}

\title
{Effective characterization of quasi-abelian surfaces}
\author{Margarida Mendes Lopes}
\address{CAMGSD/Departamento de Matem\'atica,   Instituto Superior T\'ecnico, Universidade de Lisboa,  Av. Rovisco Pais, 1049-001 Lisboa, Portugal  }
\email{\url{mmendeslopes@tecnico.ulisboa.pt}}
\author{Rita Pardini}
\address{Dipartimento di Matematica, Universit\`{a} degli studi di Pisa, Largo Pontecorvo 5, 56127 Pisa (PI), Italy}
\email{\url{rita.pardini@unipi.it}}

\author{Sofia Tirabassi}

\address{Department of Mathematics, Stockholm University, Kr\"aftriket, Stockholm, Sweden}
\email{\url{tirabassi@math.su.se}}

\newlist{todolist}{itemize}{2}
\setlist[todolist]{label=$\square$}
\usepackage{pifont}
\newcommand{\cmark}{\ding{51}}%
\newcommand{\done}{\rlap{$\square$}{\raisebox{2pt}{\large\hspace{1pt}\cmark}}%
\hspace{-2.5pt}}

\newlist{donelist}{itemize}{2}
\setlist[donelist]{label=\done}
\begin{document}

\begin{abstract}
 Let $V$ be a smooth quasi-projective complex surface such that the  first three logarithmic plurigenera $\ol P_1(V)$, $\ol P_2(V)$ and $\ol P_3(V)$ are equal to 1 and the logarithmic irregularity $\ol q(V)$ is equal to $2$. We prove that the quasi-Albanese morphism $a_V\colon V\to A(V)$ is birational and there exists a finite set $S$ such that $a_V$ is proper over $A(V)\setminus S$, thus giving a sharp effective version of a classical result of Iitaka (\cite{Ii79}).
\end{abstract}

 \maketitle
 

\tableofcontents
\section*{Introduction}

Let $V$ be a smooth complex quasi-projective  variety.   By Hironaka's theorem on the resolution of  singularities we can write $V=X\backslash D$ where $X$ is a smooth projective variety and $D$ is a reduced divisor on $X$ with simple normal crossings support (in fact in the case of surfaces, our main interest here,  $V$  is automatically  quasi-projective since it is smooth). \par
 Denoting by $K_X$ the canonical divisor  of $X$, one  defines the following invariants of $V$:
\begin{itemize}[label=-]
 \item for $m$ a positive integer, the $m$-th \emph{log-plurigenus} of $V$ is $\overline{P}_m(V):=h^0(X,m(K_X+D))$;
 \item the \emph{log-Kodaira dimension} of $V$ is $\overline{\kappa}(V):=\kappa(X,K_X+D);$
 \item the \emph{log-irregularity} of $V$ is $\overline{q}(V):=h^0(X,\Omega^1_X(\log D))$.
\end{itemize}
In addition, we say that the \emph{irregularity} $q(V)$ of $V$ is the irregularity of $X$, that is
$$q(V):=h^0(X,\Omega^1_X).
$$
It easy to see that these invariants do not depend on the choice of the compactification $X$.\par
Similarly to what happens for  projective varieties, to $V$ we can associate a quasi-abelian variety (i.e., an algebraic group which does not contain $\mathbb{G}_a$) $A(V)$, called the \emph{ quasi-Albanese variety} of $V$. This comes equipped with a morphism $a_V\colon V\rightarrow A(V)$ which is called the \emph{quasi-Albanese morphism}. Iitaka in \cite{Ii79} characterizes quasi-abelian surfaces as surfaces of log-Kodaira dimension 0 and log-irregularity 2. More precisely he proves the following.
\begin{thm*}[Iitaka, \cite{Ii79}]
{\em  Let $V$ be a smooth complex algebraic surface. Then $\overline{\kappa}(V)=0$ and $\overline{q}(V)=2$ if, and only if, the quasi-Albanese morphism $a_V\colon V\rightarrow A(V)$ is birational and there are an open subset $V^0\subseteq V$, and finitely many points $\{p_1,\ldots p_t\}\subseteq A(V)$, such that the restriction $a_{V|_{V^0}}\colon V^0\rightarrow A(V)\backslash \{p_1,\ldots p_t\}$ is proper.}
\end{thm*}

In this paper, 
  we give a characterization of quasi-abelian surfaces using the three first logarithmic plurigenera instead of the logarithmic Kodaira dimension.  Our main result is the following Theorem:

\begin{thm}\label{Main}
Let $V$ be a smooth complex algebraic surface with  $\overline{q}(V)=2$. 
Assume that:
\begin{itemize}
\item[(a)] $\overline{P}_1(V)=\overline{P}_2(V)=1$ and $q(V)>0$; or
\item[(b)] $\overline{P}_1(V)=\overline{P}_3(V)=1$ and $q(V)=0$. 
\end{itemize}

Then the quasi-Albanese morphism $a_V\colon V\rightarrow A(V)$ is birational. In addition, there are an open subset $V^0\subseteq V$ and finitely many points $\{p_1,\ldots p_t\}\subseteq A(V)$, such that the restriction $a_{V|_{V^0}}\colon V^0\rightarrow A(V)\backslash \{p_1,\ldots p_t\}$ is proper.
\end{thm}
Using the language of WWPB equivalences introduced by Iitaka in \cite{Ii77}, 
Theorem \ref{Main} above can be rephrased in the following manner:
\begin{thmB}
Let $V$ be a smooth algebraic surface. Then $V$ is WWPB equivalent to a quasi-abelian variety if, and only if  either 
\begin{itemize}
\item[(a)] $\overline q(V)=2$, $\overline{P}_1(V)=\overline{P}_2(V)=1$ and $q(V)>0$; or
\item[(b)] $\overline q(V)=2$ , $\overline{P}_1(V)=\overline{P}_3(V)=1$ and $q(V)=0$. 
\end{itemize} 
\end{thmB}
As WWPB-maps between normal affine varieties are actually isomorphisms (see \cite[Corollary on p. 498]{Ii77}), we have:
\begin{introcor}
A smooth complex affine surface $V$ is isomorphic to $\mathbb{G}_m^2$ if, and only if, it has $\overline{P}_1(V)=\overline{P}_3(V)=1$ and $\overline{q}(V)=2$.
\end{introcor}

We remark that  Kawamata, in his celebrated work \cite{Ka81}, proved in any dimension a weaker form of  Iitaka's Theorem, showing that the quasi-Albanese morphism of an algebraic variety of log-Kodaira dimension 0 and log-irregularity equal to the dimension is birational. In the compact case, effective versions of this result  have been given in  \cite{CH01} and  \cite{PPS17}: in \cite{CH01}, it is proven that the Albanese map of a projective variety $X$ is surjective and birational iff  $\dim X=q(X)$ and $P_1(X)=P_2(X)=1$; in \cite{PPS17}, the analogous statement is proven for compact K\"ahler manifolds.

So we are led to formulate the following conjecture:
\begin{conj}
 Let $V$ a smooth complex quasi-projective  variety with $\overline{q}(V)=\dim V$, then there exists a positive integer $k$, independent of the dimension of $V$, such that $\overline{P}_1(V)=\overline{P}_k(V)=1$ implies that the quasi-Albanese morphism of $V$ is birational (and there exist an open set $V^0\subseteq V$ and  a closed set $W\subseteq A(V)$ of codimension $>1$ such that $a_{V|_{V^0}}\colon V^0\rightarrow A(V)\backslash W$ is proper).
\end{conj}

 In view of the results recalled above,  one might hope  that $k=2$ is the right bound also in the open setting, but in fact our Theorem \ref{Main} is sharp because there is a surface  $V$ with  $\ol   P_1(V)=\ol P_2(V)=1$, $\ol q(V)=2$  and quasi-Albanese map not dominant (see Example  \ref{rem: sharp}).

Our argument  is completely independent of Iitaka's Theorem and its proof.
The first  step  consists in   showing that the assumptions on the logarithmic irregularity and on the logarithmic plurigenera imply that the quasi-Albanese map of $V$ is dominant. 
When $ q(V)=2$,  one obtains as an immediate consequence of the work of Chen and Hacon  (\cite{CH01}) that the quasi-Albanese map of $V$ is birational. 
 When $q(V)=0$, the lengthy proof uses a fine analysis of the boundary divisor and classical  arguments from the theory of surfaces. However, when $q(V)=1$, we get   a considerably shorter proof leveraging on the fact that the Albanese map of a compactification of $V$ is nontrivial, using techniques coming from Green--Lazarsfeld generic vanishing theorems (\cite{GL1987}) and their more recent extensions to pairs (see \cite{PS14} and \cite{Sh16}). 
 \par
Once we know that the quasi-Albanese map of $V$ is dominant, then we get that its extension to the compactifications of $V$ and $A(V)$ respectively is generically finite. Then we  use Iitaka's logarithmic ramification formula (see \ref{LogRam} for more details) together with some geometric considerations to conclude that the quasi-Albanese map is birational. \par

The last step of our argument consists in proving that the quasi-Albanese map is proper outside a finite set of points. This is done by showing that the boundary divisor $D$ gets contracted by the quasi-Albanese map. Again, the key tool here is the logarithmic ramification formula together with some more classical arguments exploiting the geometry of the divisor $D$.\par
The paper is organized as follows.  In \S \ref{sec: prelim} we recall the necessary prerequisites; \S \ref{any dim} contains some results that hold in arbitrary dimension for $q(V)>0$ (see in particular  Corollary \ref{thm:qx=2}). Then we focus on surfaces. Section \ref{sec: dominant} is devoted to proving that the quasi-Albanese map is dominant: here the results of  \S  \ref{any dim} are crucial in case $q(V)>0$. In \S \ref{sec: main} we complete the proof of Theorem \ref{Main}.

\noindent{\bf Notation.} 
We work over the complex numbers. If $X$ is  a smooth projective  variety, we denote by  $K_X$ the canonical class, by $q(X):=h^0(X,\Omega^1_X)=h^1(X,\mathcal O_X)$ the \emph{irregularity } and by  $p_g(X):=h^0(X,K_X)$ the {\em geometric genus}.

We identify invertible sheaves and Cartier divisors and we use the additive and multiplicative notation interchangeably. Linear equivalence is denoted by $\sim$. Given two divisors $D_1$ and $D_2$ on $X$, we write $D_1\succeq D_2$ (respectively $D_1\succ D_2$) if the divisor $D_1-D_2$ is (strictly) effective.

Finally note that  throughout this paper, we use the term \emph{(-1)-curve} to indicate the total transform $E$ of a point $q\in Y$ via a birational morphism $g\colon X\to Y$ of smooth projective surfaces such $g\inv$ is not defined at $q$; so one has $E^2=K_XE=-1$ but $E$ may be reducible and/or non reduced.  

\vspace{.5cm}\noindent{\bf Acknowledgments.}  
We wish to thank the referee for the extremely careful reading of the paper.

M. Mendes Lopes  was partially supported by  FCT/Portugal  through Centro de An\'alise Matem\'atica, Geometria e Sistemas Din\^amicos (CAMGSD), IST-ID, projects UIDB/04459/2020 and UIDP/04459/2020.  R. Pardini was  partially supported by  project PRIN 2017SSNZAW$\_$004 ``Moduli Theory and Birational Classification"  of Italian MIUR and is a member of GNSAGA of INDAM.   S.Tirabassi was partially supported by  grant 261756 of the Research Council of Norway and  grant KAW 2019.0493 of Knut and Alice Wallenberg 
Foundation.

\section{Preliminaries}\label{sec: prelim}
Our proof of Theorem \ref{Main} combines different  arguments and techniques; in this section  we recall briefly the necessary   prerequisites and set the notation. 
\subsection{Log varieties}

Let $X$ be a smooth complex projective variety of dimension $n$. A reduced effective divisor $D=\sum D_i$ on $X$ is said to have \emph{simple normal crossings} (in short we say $D$ is {\em snc}) if all the $D_i$ are smooth and  for every $p\in\operatorname{Supp} D$  there are local coordinates $(x_1,\ldots,x_n)$ 
around $p$ such that $D$ is cut out by the equation $x_1\cdots x_r=0$, for some  $r\leq n$.

Given a smooth projective $n$-dimensional variety $X$ together with a snc divisor $D$, 
we can define the sheaf of logarithmic 1-forms along $D$ by setting
$$
\Omega^1_X(\log D)_p:=\sum_{i=1}^r\mathcal{O}_{X,p}\frac{dx_i}{x_i}+\sum_{i=r+1}^n\mathcal{O}_{X,p}dx_i \subset(\Omega^1_X\otimes k(X))_p
$$
where $(x_1,\ldots,x_n)$ are local coordinates around $p$ such that $D=\{x_1\cdots x_r=0\}$. It is a locally free sheaf of rank equal to $n$. 
The sheaf of logarithmic $m$-forms is defined as
$$\Omega_X^m(\log D):=\bigwedge^m\Omega_X(\log D),
$$
and, in particular, we have that 
$$
\Omega_X^n(\log D)\simeq \mathcal{O}_X(K_X+D).
$$

We recall the following condition for the existence of logarithmic 1-forms with prescribed poles. 

\begin{prop} \label{prop: form}
Let   $D$ be a simple normal crossing divisor on a smooth projective variety $X$ and let $D_1, \dots D_k$ be a subset of the irreducible  components of $D$.

There exists   $\sigma\in H^0(X,\Omega^1_X(\log D))$ with poles precisely on $D_1 \dots D_k$  if and only if in $H^2(X,\mathbb Z)$ there is a relation $\sum_i a_i [D_i ]=0$  between the classes  $[D_1],\dots [D_k]$  such that  $a_i\ne0$ for $i=1,\dots k$.
\end{prop}
\begin{proof}
Let $D_1, \dots D_t$ be the irreducible components of $D$; consider the residue sequence
$$0\to \Omega^1_X\to \Omega^1_X(\log D) \to \oplus_{i=1}^t\OO_{D_i}\to 0.$$
The claim follows from the fact that the associated coboundary map $\delta\colon \oplus_{i=1}^tH^0(X,\OO_{D_i})\to H^1(X,\Omega^1_X)$ sends $1\in H^0(X,\OO_{D_i})$ to the class $[D_i]\in H^1(X,\Omega^1_X)$.
\end{proof}

Given a smooth quasi-projective variety $V$, by Hironaka's resolution of singularities we can embed it in a smooth projective variety $X$ such that $X\backslash V$ is a snc divisor $D$,  called the \emph{boundary of $V$}.
Then we can use the sheaves of logarithmic forms to define logarithmic  (in short ``log'') invariants  on $V$,  as explained in the introduction. In the case when $V$ is a curve, we call $\overline{P}_1(V)=\overline{q}(V)$ the logarithmic genus of $V$.\par

\indent While the usual plurigenera and irregularity are birational invariants, this is not the case for logarithmic invariants. For example, an abelian variety $X$ and $X\backslash D$, where  $D>0$ is a smooth ample divisor, are birational  but they have different  logarithmic Kodaira dimensions. For this reason Iitaka in \cite {Ii77} introduced the notion of WPB-equivalence ({\em weakly proper birational} equivalence). 
It is the equivalence relation generated by the proper birational morphisms and by the open inclusions $V\subset  V'$ such that $V'\setminus V$ has codimension at least 2. \par
WPB-equivalent varieties have the same plurigenera and log irregularity: this is obvious for open immersions as above and it is proven  for proper birational morphisms  in  \cite[Prop. 1]{Ii77b}. Therefore, it would seem that WPB-equivalence is the right notion of equivalence when studying the birational geometry of open varieties. However, the set of WPB-maps is not \emph{saturated}, namely  it is possible, for instance,  to have rational maps  $g\colon U\dashrightarrow V$  and $f\colon V\dashrightarrow W$ such that $f\circ g$  is WPB, but $f$ or $g$ is  not. In order to get around this kind of difficulty,  Iitaka in \cite{Ii77} defines the notion of WWPB-equivalence, proving that WWPB-equivalent varieties have the same logarithmic  plurigenera  and irregularity.

\subsection{Quasi-abelian varieties and the quasi-Albanese map}\label{qAbelian}
Here we  introduce one of the main characters of our story: quasi-abelian varieties, which for many aspects can be thought of as a non-projective analogue of abelian varieties. We recall briefly the facts that we need. 
\begin{defn} A  \emph{quasi-abelian variety} - in some sources also called a \emph{semiabelian variety} - is a connected algebraic group $G$ that is an extension of an abelian variety $A$ by an algebraic torus. More precisely, $G$ sits in the middle of an exact sequence of the form

 \begin{equation}\label{eq:sucQA}1\rightarrow \mathbb{G}_m^r\longrightarrow G\longrightarrow A\rightarrow 0.\end{equation}
 We call $A$ the {\em compact part}  and $\mathbb{G}_m^r$ the {\em linear part} of $G$.
 \end{defn}

Over the complex numbers, $G\simeq \mathbb{C}^{\dim G}/\pi_1(G)$. Observe that $\pi_1(G)$ is a finitely generated free abelian group. When the rank of $\pi_1(G)$ is equal to  $2\dim (G)$, then  $G$ is  an abelian variety.\par

For later use we recall the following from \cite[\S 10]{III}:
\begin{prop}\label{prop: compactification}
 Let $G$ be a quasi abelian variety,  let $A$ be its compact part and let $r:=\dim G-\dim A$. Then there exists a compactification $G\subset Z$ such that:
\begin{itemize}
\item[(a)] $Z$  is a $\mathbb P^r$-bundle over $A$;
\item[(b)] $\Delta:=Z\setminus G$ is a simple normal crossing divisor and  $\Omega^1_{Z}(\log \Delta)$ is a trivial bundle of rank equal to $\dim G$.
\end{itemize}
In particular, $\ol q(G)=\dim G$ and $\ol P_m (G)=1$ for all $m>0$.
\end{prop}

We are especially interested in the following particular case of Proposition \ref{prop: compactification}:
\begin{cor} \label{cor: Z}
Let $G$ be a quasi-abelian variety of dimension 2 with compact part $A$ of dimension 1. Then there is a compactification $G\subset Z$,  where $Z=\mathbb P (\OO_A\oplus L)$ with $L\in \Pic^0(A)$ and the boundary  $\Delta$ is the disjoint union of two sections $\Delta_1$ and $\Delta_2$ of $Z\to A$. 
\end{cor}
\begin{proof} By Proposition \ref{prop: compactification} we may find a compactification $Z$ which is a $\mathbb P^1$-bundle over $A$ and such that $\Delta$ is a normal crossing divisor. The divisor $\Delta$ meets every fibre of $Z\to A$ in exactly two points, so it is a smooth bisection. By Proposition \ref{prop: compactification}, there is a  non regular logarithmic 1-form of $Z$   with poles contained in $\Delta$, hence  by Proposition \ref{prop: form} $\Delta$ is reducible and its components satisfy a non trivial relation in cohomology. Since $\Delta$ is a smooth bisection of $Z\to A$, we have $\Delta=\Delta_1+\Delta_2$ with $\Delta_i$ disjoint sections. So $Z=\pp (\OO_A\oplus L)$ for some $L\in \Pic(A)$.  Since  in $\Pic(Z)$ we have 
$\Delta_2=\Delta_1+p^*L$, where $p\colon Z\to A$ is the natural projection, $\Delta_1$ and $\Delta_2$ are independent in $H^2(Z,\mathbb Z)$ unless $p^*L=0$ in $H^2(Z,\mathbb Z)$, namely unless $p^*L\in \Pic^0(Z)$. So  Proposition \ref{prop: form} implies that $L$ is an element of $\Pic^0(A)$.\end{proof} 

Finally we  recall some  fundamental properties of abelian varieties that extend verbatim to the quasi-abelian case:
\begin{prop} \label{prop: quasiab}
Let $G$ be a quasi-abelian variety. Then:
\begin{enumerate}
\item if $G'$ is a quasi-abelian variety and $\phi\colon G'\to G$ is a morphism with $\phi(0)=0$, then $\phi$ is a homomorphism;
\item if $G'\to G$ is a finite \'etale cover, then $G'$ is a quasi-abelian variety;
\item if $H\subset G$ is closed and $\ol\kappa(H)=0$ then $H$ is a quasi-abelian variety.
\end{enumerate}
\end{prop}
\begin{proof} 
Item (i) is \cite[Thm.~5.1.37]{Nevanlinna}, (ii) is \cite[Thm.~4.2]{Fu14}, and  (iii) is \cite[Thm.~ 4]{III}.
\end{proof}
 \subsubsection{The quasi-Albanese map} 
 The classical construction of the Albanese variety of a projective variety can be extended to the non projective case, by replacing regular 1-forms by  logarithmic ones and  abelian varieties by quasi-abelian ones. The key fact is that by Deligne \cite{De71}  logarithmic 1-forms are closed (for the details of the  construction see \cite{III}, \cite[Section 3]{Fu14}).

 \begin{theorem}\label{thm: quasiA}
 Let  $V$ be a smooth algebraic variety. Then there exists a quasi-abelian variety $A(V)$ and a morphism $a_V\colon V\to A(V)$ such that:
 \begin{enumerate}
 \item if $h\colon V\rightarrow G$  is a morphism to a quasi-abelian variety, then $h$ factors through $a_V$ in a unique way;
 \item if $X$ is a compactification of $V$ with snc  boundary $D$, then  we
have the following exact sequence:
\begin{equation}\label{eq:ses}1\rightarrow \mathbb{G}_m^r\longrightarrow A(V)\longrightarrow A(X)\rightarrow 0,
\end{equation}
where $r=\ol q(V)-q(V)$. 
 \end{enumerate} 
 \end{theorem}
 \begin{proof}
 Item (i) is proven in \cite{III}, in the discussion immmediately after Proposition 4. Item  (ii) is  \cite[p. 13]{Fu14}
 \end{proof}
 The variety $A(V)$ is called the {\em quasi-Albanese variety} of $V$ and $a_V$ the {\em quasi-Albanese map}. Note that  the compact part of  $A(V)$ is $A(X)$. 
 We note the following logarithmic version  of Abel's Theorem:
 \begin{prop}\label{prop: log-abel}
  Let $C$ be a smooth curve with $\ol P_1(C)>0$. Then the quasi-Albanese map $a_C\colon C\to A(C)$  is an embedding. 
 \end{prop}
 \begin{proof} Denote by $\ol C$ the compactification of $C$. If $g(\ol C)>0$, then the Abel-Jacobi map $\ol C\to A(\ol C)$ factorizes through $a_C$, so the claim follows by Abel's Theorem. 
 If $g(\ol C)=0$, we have $C:=\ol C\setminus \{p_0,\dots p_k\}$, where $k:=\ol P_1(C)$.  By the universal property, the inclusion $C\to \ol C\setminus \{p_0,p_1\}\simeq \mathbb G_m$ factorizes through $a_C$, which therefore is an isomorphism.
  \end{proof}

\subsection{Logarithmic ramification formula }\label{LogRam}
Let $V$, $W$ be smooth varieties of dimension $n$ and let  $h\colon V\to W$ be a dominant morphism. Let $g\colon X\to Z$ be a morphism extending $h$, where $X$, $Z$ are smooth compactifications of $V$, respectively $W$, such that  $D:=X\setminus V$ and $\Delta:=Z\setminus W$ are snc divisors. Then the pull back of a  logarithmic $n$-form on $Z$ is a  logarithmic $n$-form on $X$, and a local computation shows that there is an effective divisor $\ol R_g$ of $X$ - the {\em logarithmic ramification divisor} - such that the following linear equivalence holds:
\begin{equation}\label{eq: logram}
K_X+D\sim g^*(K_Z+\Delta )+\ol R_g.
\end{equation} 
Equation \eqref{eq: logram} is called the {\em logarithmic ramification formula} (cf. \cite[\S 11.4]{Ii82}). 

 We note the following for later use:

\begin{lm}\label{lem: Rg} In the above set-up, denote by $R_g$ the (usual) ramification divisor of $g$. Let   $\Gamma$ be  an irreducible divisor   such  that $g(\Gamma)\not \subseteq \Delta$.
Then $\Gamma$ is a component of  $\ol R_g$ if and only if $\Gamma$ is a component of $D+R_g$.
\end{lm}
\begin{proof} Let $x\in \Gamma$ be a general point, so that $y=g(x)$ does not lie on $\Delta$. Then $K_Z+\Delta$ is generated locally near $y$ by a nowhere vanishing  regular $n$-form, while $K_X+D$ is generated locally near $x$  either by an $n$-form with a logarithmic pole on $\Gamma$ or by  a regular $n$-form, according to whether $\Gamma$ is a component of $D$ or not. In the former case $\Gamma$ is always a component of $\ol R_g$, in the latter case it is a component of $\ol R_g$ if and only if $g$ is ramified along $\Gamma$.
\end{proof}
\subsection{Generic vanishing}\label{ssec: GV}
 The theory of generic vanishing was introduced by Green--Lazarsfeld in \cite{GL1987}.  It has since  developed in a powerful tool to study the geometry of projective varieties via their Albanese morphism (see, for example, \cite{Pa12} for a nice survey). We are going to see in Section \ref{any dim} that these techniques can be useful instruments also in the quasi-projective setting.\par

Let $X$ be a smooth projective variety. Given  a coherent sheaf  $\mathcal{F}$ on $X$, the \emph{cohomological support loci} of $\mathcal{F}$ are the subsets
$$V^i(X,\mathcal{F}):=\{\alpha\in\Pico{X}\ |\ h^i(X,\mathcal{F}\otimes\alpha)\neq 0\}\subseteq \Pico{X}.$$

 We say that $\mathcal{F}$ is a \emph{$GV$-sheaf} if, for every $i>0$, we have that
 $$\mathrm{codim}_{\Pico{X}}V^i(X,\mathcal{F})\geq i.$$
 When all the $V^i(X,\mathcal{F})$ are empty for $i>0$, one says, following the original terminology of Mukai (\cite{Mu1981}), that $\mathcal{F}$ is a \emph{IT(0)-sheaf}.\par

We recall  the following well known useful observation:

\begin{lm} \label{lem: -Z}
Let $X$ be a smooth projective variety and $L$ a line bundle of $X$. 

If $V^0(X,L)\cap (-V^0(X,L))$ has  positive dimension, then $h^0(X, 2L)\ge 2$. 
\end{lm}
\begin{proof}
For $\alpha\in V^0(X,L)\cap (-V^0(X,L))$ consider the multiplication map
 $$H^0(X,L\otimes\alpha)\otimes H^0(X,L\otimes\alpha^{-1})\longrightarrow H^0(X,2L).$$
 As $\alpha$ varies, 
   the image of the map must vary,  since a  divisor  can be written as the sum of two effective divisors only in finitely many ways.
As a consequence $h^0(X,2L)\ge 2$.
\end{proof}

\subsection{Curves on smooth surfaces} 
Let  $D\succ 0$ be  a divisor (a ``curve'') on a smooth projective surface $X$.  The \emph{arithmetic genus} of $D$  is $p_a(D):=1-\chi(\OO_D)$ and can be computed by means of  the {\em adjunction formula} 
\begin{equation}\label{eq: adj} 
p_a(D)-1=\frac 12 D(K_X+D).
\end{equation}
By Serre duality, one also has $p_a(D)-1= h^0(D,\omega_D)-h^0(D, \OO_D)$, hence $h^0(D, \omega_D)\ge p_a(D)$. 

For  $m\in \mathbb N$, we say that $D$ is \emph{$m$-connected}  if for every decomposition $D=D_1+D_2$, with $D_1$ and $D_2$ effective, one has $D_1D_2\ge m$.\

We recall  some  well known facts. 
\begin{lm}\label{lem: pa0} Let  $D$ be  a $1$-connected divisor,  then:
\begin{enumerate}
 \item $h^0(D,\OO_D)=1$ (so  
in particular, $p_a(D)=h^0(D, \omega_D)\ge 0$);
\item if $L$ is a line bundle on $D$ that has degree 0 on every component of $D$, then $h^0(D,L)\le 1$, and $h^0(D,L)=1$ if and only if $L= \OO_D$.
\end{enumerate}
\end{lm}
\begin{proof} See \cite[Lem.~II 12.2]{BHPV}.
\end{proof}

\begin{lm} \label{pa sum} Let $D_1$, $D_2$ be two effective non zero divisors on a smooth projective  surface $X$. Then 
\begin{enumerate}
\item $p_a(D_1+D_2)=p_a(D_1)+p_a(D_2)+D_1D_2-1$;
\item if $D_1\prec D_2$, then $h^0(D_1,\omega_{D_1})\le h^0(D_2, \omega_{D_2})$.
\end{enumerate}
\end{lm}
\begin{proof}
(i) Use the adjunction formula \eqref{eq: adj}. 

(ii) Write $D_2=D_1+A$ and consider the decomposition sequence:
$$0\to \OO_{D_1}(-A)\to \OO_{D_2}\to \OO_A\to 0.$$
Twisting by $K_X+D_2$ and taking cohomology gives the result.
\end{proof}

\begin{lm} \label{lem: omegaD} Let $D$ be an effective non zero divisor on a smooth projective surface $X$. Then we have the following formulae:

\begin{enumerate}
\item $h^0(X, K_X+D)= p_a(D)+ p_g(X)-q(X)+h^1(X,-D)$;
\item $h^0(X, K_X+D)= h^0(D,\omega_D)+ p_g(X)-q(X)+d$,  \par where  $d$ is the dimension of the kernel  of the restriction map  $H^1(X, \OO_X)\to H^1(D, \OO_D)$.
\end{enumerate}
\end{lm}

\begin{proof}
(i) follows from  the Riemann-Roch theorem, the adjunction formula \eqref{eq: adj} and Serre duality.

 Taking cohomology in the restriction sequence
$$0\to\OO_X(-D)\to \OO_X\to \OO_D\to 0,$$
one obtains $h^1(X,-D)= h^0(D, \OO_D)-1+d$. Plugging this relation and the equality $p_a(D)-1= h^0(D,\omega_D)-h^0(D, \OO_D)$ in (i), one obtains (ii). 
\end{proof}

\begin{rmk} \label{codimension}Note that  by Serre duality  $d$ is exactly the codimension of  the image of  the map $t\colon H^0(D, \omega_D)\to H^1(X, K_X)$.
\end{rmk}

\section{Results in arbitrary dimension via generic vanishing} \label{any dim}

Let $V$ be a smooth quasi-projective variety and let $X$ be a smooth compactification of $V$ with snc boundary.
In this section we investigate the geometric constraints imposed on the Albanese map $a_X$ of $X$, and on the image of $V$ via $a_X$,  by the fact  the logarithmic plurigenera of $V$ are small. Our main tool will be the theory of generic vanishing (\S \ref{ssec: GV}), and its recent extensions to the ``log-canonical''  setting by Popa--Schnell (\cite{PS14})and Shibata (\cite{Sh16}). 

 \begin{prop}\label{prop:axsu}
Let $V$ be a smooth quasi-projective  variety and  denote by $X$ a compactification of $V$ with simple normal crossing boundary divisor $D$.

  If \ $\ol P_1(V)=\overline{P}_{2}(V)=1$, then the Albanese morphism  $a_X\colon X\to A(X)$ is surjective.
\end{prop}
\begin{proof}   Since $V^0(X,\mathcal{O}_X(K_X+D))$ is a union of translates of subtori of $A(X)$ (cf. \cite{Sh16}*{Theorem 1.3}), Lemma \ref{lem: -Z} implies that $\sO_X$ is an isolated point of $V^0(X,\sO_X(K_X+D))$. Thus by the projection formula $\sO_{A(X)}$ is an isolated point of $V^0(A(X),a_{X,*}\sO_X(K_X+D))$. Thanks to \cite{PS14}*{Variant 5.5}, we know that $a_{X,*}\sO_X(K_X+D)$ is a $GV$-sheaf on $A(X)$, and therefore  by \cite{Pa12}*{Lemma 1.8} we know that $\sO_{A(X)}$ is a component of $V^{q(X)}(A(X),a_{X,*}\sO_X(K_X+D))$. In particular $h^{q(X)}(A(X),a_{X,*}\sO_X(K_X+D))\neq 0$ and we deduce that the dimension of the Albanese image of $X$ is equal the dimension of $A(X)$. We conclude because $X$ and $A(X)$ are projective.
\end{proof}

As an immediate consequence  we have:
\begin{cor}\label{thm:qx=2}
 Let $V$ an $n$-dimensional smooth quasi-projective  variety,   $X$ a compactification of $V$  such that $D:=X\setminus V$ is a snc divisor. If $\overline{P}_1(V)=\overline{P}_2(V)=1$ and $\overline{q}(V)=q(X)=n$, then the quasi-Albanese morphism $a_V\colon V\rightarrow A(V)$ is  birational. 
 \end{cor}
\begin{proof}

By Proposition \ref{prop:axsu}, we know that $a_X$ is generically finite, and so $X$ is of maximal Albanese dimension. In particular, we have that $$0<h^0(X,K_X)\leq h^0(X,K_X+D)=1.$$
We conclude that $h^0(X,K_X)=1$. Similarly, we see that $h^0(X,2K_X)=1$ and so by the characterization Theorem of Chen--Hacon (\cite{CH01}) we conclude that $a_X$ is a birational morphism. In addition, by Theorem \ref{thm: quasiA}  there is a commutative diagram
\[
 \xymatrix{V\ar@{^(->}[d]\ar[rr]^{a_V}&&A(V)\ar[d]^\simeq\\
 X\ar[rr]^{a_X}&&A(X)}
\]
and so also $a_V$ is birational.\par

\end{proof}


The next result refines Proposition \ref{prop:axsu}:
\begin{prop}\label{prop:not effective}
Let $V$ be a smooth quasi-projective  variety,   $X$ a compactification of $V$  such that $D:=X\setminus V$ is a snc divisor. Let $a_X\colon X\to A(X)$ be the Albanese morphism  and 
let $E>0$ be a divisor on $A(X)$. 

If  $\ol P_1(V)>0$ and $D$ contains the support of $a_X^*E$, then $\ol P_2(V)\ge 2$.
\end{prop}
\begin{proof} 
The divisor $E$ is of the form $\pi^*H$, where  $\pi\colon A\to B$ is a morphism onto a positive dimensional abelian variety and $H$ is an ample divisor on $B$. Set $f:=\pi\circ a_X$;
by assumption there is a positive integer $N$ such that $ND\succeq f^*H$. Set $\Delta:=D-\frac{1}{N}f^*H$:  by assumption the $\mathbb Q$-divisor  $\Delta$ has snc support and $\Delta=\sum _id_i\Delta_i$, with $\Delta_i$ irreducible divisors and  $0\le d_i\le 1$.  
Given  $\alpha\in \Pic^0(B)$, the divisor $K_X+D+f^*(\alpha)$ is $\mathbb Q$-linearly equivalent 
to $K_X+\Delta+f^*(\frac
{1}{N}H+\alpha)$. Since $\frac{1}{N}H+\alpha$ is ample if $H$ is,  the assumptions of \cite[Thm.~6.3]{Fu11} or \cite[Thm.~3.2]{Ambro03} are satisfied and we have $$H^i(B, f_*\mathcal O_X(K_X+D)\otimes\alpha)=0, \quad\text{for all } i>0, \quad\text{and all }  \alpha \in \Pic^0(B).$$ So $f_*(K_X+D)$ is an IT(0)-sheaf, and therefore for all $\alpha \in \Pic^0(B)$ we have   $h^0(B, f_*\mathcal O_X(K_X+D)\otimes\alpha)=h^0(B, f_*(K_X+D))=\ol P_1(V)>0$. So $V^0(X, K_X+D)$ contains the positive dimensional abelian subvariety $f^*\Pic^0(B)=\pi^*\Pic^0(B)$ and $\ol P_2(V)\ge 2$ by Lemma \ref{lem: -Z}.
\end{proof}
 
When  $q(V)$ is equal to 1,  Proposition \ref{prop:not effective} gives: 
\begin{cor}\label{cor: albdim1}
  Let $V$ be a smooth quasi-projective  variety,   $X$ a compactification of $V$  such that $D:=X\setminus V$ is a snc divisor. Let $a_X\colon X\to A(X)$ be the Albanese morphism.  
  
  If $q(V)=\ol P_1(V)=\ol P_2(V)=1$, then $a_X(V)=A(X)$.
  \end{cor}

\begin{rmk}\label{rmk:nodiv}
 Observe that, by Proposition \ref{prop:axsu}, we can deduce that if $\ol P_1(V)=\ol P_2(V)=1$ then  the map $ a_X|_V$ is dominant.
  In addition, as a consequence of Proposition  \ref{prop:not effective}, we know that the complement of  $a_X(V)$ in $A(X)$ does not contain a  divisor. 
  Notice that, if $q(X)>1$, this does not mean that the complement of $a_X(V)$ has codimension >1, since the image of $V$ is not necessarily open in $A(X)$. 
  \end{rmk}

\section{The geometry of the quasi-Albanese morphism} \label{sec: dominant}
 
 The aim of this section is to establish the following fundamental step in the proof of Theorem \ref{Main}:
 \begin{prop}\label{prop: dominant}
 Let $V$ be a smooth complex algebraic surface with  $\overline{q}(V)=2$. 
Assume that:
\begin{itemize}
\item[(a)] $\overline{P}_1(V)=\overline{P}_2(V)=1$ and $q(V)>0$; or
\item[(b)] $\overline{P}_1(V)=\overline{P}_3(V)=1$ and $q(V)=0$. 
\end{itemize}
Then the quasi-Albanese morphism $a_V\colon V\rightarrow A(V)$ is dominant.
 \end{prop}
 
We use different approaches for the case $q(V)=0$ and $q(V)>0$, so we treat them separately. In fact,  Proposition \ref{prop: dominant} is just the combination of Propositions \ref{prop: dominant1} and \ref{prop: dominant0} below. 

\subsection{The quasi-Albanese map when $q(V)\ge 1$}

\begin{prop}\label{prop: dominant1}
Let $V$ be  a smooth algebraic surface such that $\overline{P}_1(V)=\overline{P}_2(V)=1$, $\overline{q}(V)=2$, and $q(V)\ge 1$. Then the quasi-Albanese morphism of $V$ is dominant, and hence generically finite.
\end{prop}
\begin{proof}
When $q(V)=\ol q(V)=2$ the map  $a_V$ is dominant  by Corollary \ref{thm:qx=2}.

Assume now that $q(V)=1$ and, by contradiction, that the image of $a_V$ is a curve $C$.  Denote by $\overline C$ the smooth projective model of $C$; since $\ol C$ dominates the elliptic curve $A(X)$, the curve $\overline  C$ has positive genus and $a_V$ extends to a morphism $f\colon X \to \ol C$ such that the Albanese morphism factorizes through $\overline C$. By the universal property of the Albanese map, the map $\ol C\to A(X)$ is an isomorphism. Now Corollary \ref{cor: albdim1} implies that $C=\overline C$. In particular, by Proposition \ref{prop: quasiab}  up to translation $C$ is an algebraic subgroup of $A(V)$. On the other hand, $C$ must generate $A(V)$ as an algebraic group, so it follows $C=A(V)$, a contradiction. 
 \end{proof}

\subsection{The quasi-Albanese map when $q(V)=0$}
In this section we prove the following
\begin{prop}\label{prop: dominant0}
Let $V$ be  a smooth algebraic surface such that $\overline{P}_1(V)=\overline{P}_3(V)=1$, $\overline{q}(V)=2$, and $q(V)=0$. Then the quasi-Albanese morphism of $V$ is dominant, and hence generically finite.
\end{prop}
The proof is based on numerical arguments and is quite intricate, so we break it into several steps. We start by proving  two  results on curves on smooth projective surfaces. 
\medskip 
\begin{lm} \label{A} Let $X$ be a non-singular projective surface  and $A$ a 1-connected effective divisor with $p_a(A)=1$. 

Then $A$ contains a 2-connected divisor $B$  such that $p_a(B)=1$.
\end{lm}

\begin{proof} If $A$ is  2-connected there is of course nothing to prove. Otherwise,  take any decomposition  $A=A_1+A_2$ with $A_1A_2=1$ and $A_1$ minimal with respect to $A_1A_2=1$. $A_2$ is 1-connected,   $A_1$ is 2-connected and by  Lemma \ref{pa sum} $p_a(A_1)+p_a(A_2)=1$, $p_a(A_1)\geq 0$,  $p_a(A_2)\geq 0$. If $p_a(A_1)=1$ we have proved the statement, if not we repeat the argument on $A_2$.   
\end{proof}

\begin{lm} \label{B} Let $X$ be a non-singular projective surface  and $B$ a 2-connected effective divisor with $p_a(B)=1$. Then:
\begin{enumerate}
\item  if $B$ is not irreducible then every component $\Gamma$ of $B$ is smooth rational and  satisfies $(K_X+B)\Gamma=0$;
\item $\omega_B=\mathcal O_B$;
\item if  $\Gamma$ is an irreducible component of $B$ and $B-\Gamma\succ 0$, then  $B-\Gamma$ is  1-connected.
\end{enumerate} 
\end{lm}

\begin{proof} 
  (i) Since $p_a(B)=1$, $(K_X+B)B=0$.  On the other hand for every  irreducible component $\Gamma$ of $B$ one has $(K_X+B)\Gamma=(K_X+\Gamma)\Gamma+ (B-\Gamma)\Gamma\geq 0$ because $(K_X+\Gamma)\Gamma\geq -2$ by adjunction and  $(B-\Gamma)\Gamma\geq 2$ since  $B$ is 2-connected and reducible.  So  necessarily  $(K_X+B)\Gamma=0$, and $\Gamma$ is smooth rational.%
  
  (ii) Since, by (i)  $\omega_B$ has degree $0$ on every component of $B$ and $h^0(B, \omega_B)=1$, by Lemma \ref{lem: pa0}  $\omega_B=\OO_B$.
   
  (iii)  By  the proof of (i) one has $\Gamma(B-\Gamma)=2$. Let $B-\Gamma=A_1+A_2$ with $A_1>0,A_2>0$. Then because $B$ is 2-connected $A_i(A_j+\Gamma)\geq 2$ for $\{i,j\}=\{1,2\}$. Since $(A_1+A_2)\Gamma=2$ necessarily $A_1A_2\geq 1$ and so $B-\Gamma$ is 1-connected. 
 \end{proof} 

\begin{lm}\label{sumE}
 Let $X$ be a non-singular projective surface with $q(X)=0$ and let $B$ be an effective 2-connected divisor satisfying   $p_a(B)=1$. Then one of the following occurs: 
\begin{enumerate}

\item[(a)]    $h^0(X, 2K_X+2B)\geq 2$;

\item[(b)]    $X$ is rational,  and  there is a blow-down morphism $\rho\colon X\to T$  with exceptional divisor $\sum_{j=1}^nE_j$  (where the $E_j$ are $-1$-curves) such that   
$K_X+B\sim \sum_{j=1}^nE_j$ and $B$ is disjoint from $\sum_{j=1}^nE_j$.
\end{enumerate} 
\end{lm}

\begin{proof}
 From Lemma \ref{lem: omegaD} we obtain $h^0(X, K_X+B)=p_g(X)+1\geq 1$.

Assume that $K_X+B$ is nef. Then since $(K_X+B)B=0$, we have $K_X(K_X+B)=(K_X+B)^2\geq 0$ and 
$(2K_X+B) (K_X+B)=2K_X(K_X+B)\geq 0$.  So we have two possibilities: 
either $K_X+B\sim 0$  (and $p_g(X)=0$) or   there is a non zero effective divisor   $B_1$ in $|K_X+B|$,  and $p_a(B_1)\geq 1 $  implying by Lemma \ref{lem: omegaD} that  $h^0(X, K_X+B_1)\geq 1+p_g$, i.e. $h^0(X, 2K_X+B)\geq 1$. 

In the second case  the restriction map 
 $H^0(X, K_X+B)\to H^0(B, \omega_B)$ is surjective because $q(X)=0$. Since   $\omega_B=\mathcal O_B$, the map  $ H^0(X, 2K_X+2B)\to H^0(B, \omega_B^{\otimes 2})$ is also nonzero. So the  exact sequence 

$$0\to H^0(X, 2K_X+B)\to H^0(X, 2K_X+2B)\to H^0(B, \omega_B^{\otimes 2})$$  gives  $h^0(X, 2K_X+2B)\geq 2$.
  \medskip

  Assume now that $B_1$ is not nef and let $\theta$ be an irreducible curve with $B_1\theta<0$. Since  $B_1$ is effective,  $\theta$ is a component of $B_1$ with $\theta^2<0$.  In addition, for every component $\Gamma$ of $B$,  we have, by Lemma \ref{B},  $B_1\Gamma= 0$, so $\theta$ is not a component of $B$,  and thus $B\theta \ge 0$. Then the only possibility is that $\theta$ is an irreducible  $-1$-curve disjoint from $B$. We contract $\theta$ and replace $B$ by its image under the contraction; repeating this process we eventually end up with a birational morphism  of smooth surfaces $ \rho\colon  X\to T$  such that 
$K_{T}+\rho(B)$ is nef, and $\rho(B)$ is still a 2-connected divisor with $p_a=1$.

By the discussion in the previous case, we see that either  $h^0(T, 2K_T+2\rho(B))\ge 2$ or  $K_T+\rho(B)\sim 0$. In the former case $h^0(X, 2K_X+2B)\ge h^0(T, 2K_T+2\rho(B))\ge 2$. In the latter  case,   taking pull-backs we get  $\rho^*(K_T)+B=K_X-\sum_{j=1}^nE_j+B \sim 0$. So, if $L$ is an ample divisor on $T$, then $\rho^*L$ is nef and big and it satisfies $K_X\rho^*L=-B\rho^*L<0$, so $\kappa(X)=-\infty$. Since $q(X)=0$, the surface $X$ is rational. 
\end{proof}

For the rest of the section, we  refer to  the following situation:
\begin{set}\label{setup}
We let $V$ be a smooth open algebraic surface with $\overline{q}(V)=2$, $q(V)=0$, and $\ol P_1(V)=\ol P_3(V)=1$. We consider  the standard compactification $Z=\mathbb P^1\times \mathbb P^1$ of  $A(V)=\mathbb G_m^2$,  and we denote by $\Delta:=Z\setminus A(V)$ the boundary. Finally, we fix a compactification $X$ of $V$ with snc boundary $D$ such that the quasi-Albanese map $a_V\colon V\to A(V)$ extends to a morphism $g\colon X \to Z$. 
\end{set}

We exploit the previous results to gain more  information on the pair $(X,D)$.
\begin{lm}\label{lem: B}
 There is a blow down  morphism $\rho\colon X\to T$   with exceptional divisor $\sum_{j=1}^n E_j$ (where the $E_j$ are $-1$-curves),  such that one of the following holds:
\begin{enumerate}
\item[(a)] $p_g(X)=0$, $h^0(D,\omega_D)=1$,   $X$ is a rational surface and there is a 2-connected divisor $B\preceq D$ such that $K_X+B\sim \sum_{j=1}^n E_j$ and $B$  is disjoint from $\sum_{j=1}^n E_j$;
\item[(b)]  $p_g(X)=1$, $h^0(D,\omega_D)=0$, and  $T$ is a K3 surface.
\end{enumerate}
In either case there is a  divisor $B\succeq 0$  such that $K_X+B\sim \sum_{j=1}^n E_j$. 
\end{lm}
\begin{proof}
The assumption  $\ol P_1(V)=1$   implies that either $p_g(X)=0$ or $ p_g(X)=1$; in either case the value of $h^0(D,\omega_D)$ can be computed by Lemma \ref{lem: omegaD}.

Assume first $p_g(X)=0$. Then, by Lemma \ref{lem: omegaD},  $q(X)=0$ implies that   $h^0(D,\omega_D)=1$. So there is a connected component $A$ of $D$ such that $p_a(A)=h^0(A,\omega_A)=1$ and by 
 Lemma \ref{A} there  is  a 2-connected divisor $B\preceq A$ with $p_a(B)=1$.  By Lemma  \ref{sumE}   
and the hypothesis $\ol P_2(V)=1$, $X$ is rational and there is a blow down  morphism $\rho\colon X\to T$   as in (a).

Assume now $p_g(X)=1$. In this case  $X$ has non-negative Kodaira dimension and 
 we take  $\rho\colon X\to T$ to be the morphism to the minimal model. 
If $X$ is of general type,  $K_T^2>0$ and  $h^0(X, 2K_X)=h^0(T, 2K_T)=K_T^2+\chi(\mathcal O_X)\geq 2$,  contradicting $\ol P_2(V)=1$.


 If $X$ is properly elliptic then, because $p_g(T)=1$ and $K_T\neq \mathcal O_T$, $h^0(T, -K_T)=0$, and so by duality $h^2(T, 2K_T)=0$. Since $K_T^2=0$, by the Riemann-Roch theorem  we obtain $h^0(T, 2K_T)=h^1(T,2K_T)+\chi(\mathcal O_X)\geq 2$.

So $\kappa(X)=0$ and by  the classification of projective surfaces we conclude that $T$ is a $K3$-surface. We have $K_X=\sum_{j=1}^n E_j$, so in this case the last claim holds with $B=0$.
\end{proof}

\begin{lm}\label{ho}    One has:
\begin{enumerate}
\item
$h^0(X, 2K_X+D)=h^0(X, 3K_X+2D)=p_g(X)$;
\item  if $D_i$ is the unique divisor in $|i(K_X+D)|$, $i=1,2$, then  $h^0(D_i, \omega_{D_i})=0$.
\end{enumerate}
  \end{lm} 
\begin{proof}
(i) Since $D\succ  0$ and   $\ol P_i(V)=1$ for $i=1,2,3$ by assumption, we have $p_g(X)\le h^0(X, 2K_X+D)\le h^0(X, 3K_X+2D)\le 1$.  So  if   
$p_g(X)=1$ the assertion is trivial. 

For  $p_g(X)=0$  and $m\ge 1$ consider the exact sequence: 
\begin{equation}\label{eq: ho}
0\to H^0(mK_X+(m-1)D)\to H^0(X, m(K_X+D))\to H^0 (D, \omega_D^{\otimes m})
\end{equation}
 The second map in \eqref{eq: ho} is an isomorphism for $m=1$, since $q(X)=0$, so it is nonzero for all $m\ge 1$, since $D$ is a reduced divisor. 
Then $\ol P_2(V)=1$  and  $\ol P_3(V)=1$ imply  that $h^0(X, 2K_X+D)=h^0(X, 3K_X+2D)=0$.

(ii)  is an immediate consequence of  $q(X)=0$ and  of (i)  (see Lemma \ref{lem: omegaD}).

 \end{proof}

We now turn to the study of the quasi-Albanese map.
\begin{lm} \label{lem: 3fibre}
If the image of $a_V$  is a curve, then it  is isomorphic to $\mathbb P^1\setminus \{0,1,\infty\}$ and the general fiber of $a_V$ is connected.
\end{lm} 
\begin{proof}
Let $g$ as in Setting \ref{setup}.
By assumption, the  image $\Gamma$ of $g$ is a curve. Let $X\to \widetilde \Gamma\to \Gamma$ be the Stein factorization of $g$ and let $\Gamma_0\subset \widetilde\Gamma$ be the image of $V$. Then by the universal property of the quasi-Albanese map, $A(V)$ is isomorphic to $A(\Gamma_0)$. In addition, the image of $a_V$ is isomorphic to $\Gamma_0$ by Proposition \ref{prop: log-abel}.
It follows that $g$ has connected fibers, hence the general fiber of $a_V$ is also connected.  Finally $\Gamma_0$ is rational, since $q(X)=0$, and has logarithmic genus 2, hence it is  isomorphic to $\mathbb P^1\setminus \{0,1,\infty\}$.\end{proof}

We get immediately:
\begin{cor}
If the image of $a_V$ is a curve, then there is a fibration $f\colon X\to\mathbb P^1$ such that $D$ contains  the supports $F_1^s$, $F_2^s$ and $F_3^s$ of three distinct  fibers  $F_1$, $F_2$ and $F_3$ of $f$.
\end{cor}

\begin{lm}\label{ei}  Let $E_j$, $j=1,\dots n$, be  the $-1$-curves contracted by   the blow down morphism  $\rho \colon X\to T$ of Lemma \ref{lem: B}. 

Then  $ E_jF_i^s\geq 0$ for all $j,i$.
\end{lm} 
\begin{proof}
 Assume by contradiction that $E_jF_1^s< 0$. Then $E_j$ and $F_1^s$ have common components.    Write  $E_j=A+C_1$ and $F_1^s=A+C_2$ where $A\succ 0$ and $C_1,C_2
\succeq 0$ have no common components.  Note that  $C_2\succ 0$ otherwise blowing down $E_j$  we would contract the whole fiber.  Now $E_jF_1^s< 0$ yields $A^2+AC_1+AC_2+C_1C_2<0$. But, because $E_j$ is a $(-1)$-curve,  $E_jA=A^2+AC_1\geq -1$ (see \cite[Prop.~3.2]{ADE}) and  because $F_1^s$ is 1-connected $AC_2\geq 1$. Since $C_1C_2\geq 0$ we have a contradiction. 
 \end{proof}

 \begin{lm}\label{components} Let $E_j$, $j=1,\dots n$, be  the $-1$-curves contracted by   the blow down morphism  $\rho \colon X\to T$ of Lemma \ref{lem: B}. 
 
 Then   $\sum_{j=1}^n E_j+D$ has at most one component transversal to $f$.\end{lm} 
\begin{proof}
Let $ M_1,M_2$ be two distinct components of $\sum_{j=1}^n E_j+D$ tranversal to $f$. Then  because $F_i^s$ is the support of a full fiber  $M_k(\sum_{i=1}^3 F_i^s)\geq 3$ for $k=1,2$  yielding $p_a(M_1+M_2+\sum_{i=1}^3 F_i^s) \geq 2$ and so,  by Lemma \ref{lem: omegaD},  $h^0(X, K_X+M_1+M_2+\sum_{i=1}^3 F_i^s)\geq 2$.

By Lemma \ref{lem: B} there is a divisor $0\preceq B\preceq D$ such that $K_X+B\sim \sum_j E_j$.
Since  $M_1+M_2+\sum_{i=1}^3 F_i^s \leq \sum_{j=1}^n E_j+D\in |K_X+B+D|$   we obtain   $h^0(X, 2K_X+B+D)\geq 2$, which contradicts $\ol P_2(V)=1$. 

\end{proof}

\begin{cor}\label{common}  Assume $p_g(X)=0$, and let $0\prec B\preceq D$ be the divisor such that $K_X+B\sim \sum_j E_j$ (cf. Lemma \ref{lem: B}).
If $B$ has components in common  with $\sum_{i=1}^3 F_i^s$, then one of the following happens: 
\begin{enumerate}

\item $B\leq F_i^s$ for some  $i\in \{1,2,3\}$, or ;
\item $B$ has a unique component $H$ transversal to $f$  and $B-H$ is contained in, say,  $F_1^s$. Furthermore $HF_1^s=2$ and $HF_i^s=1$ for $i=2,3$.
\end{enumerate}
\end{cor}

\begin{proof}

 Suppose no component of $B$   is transversal to $f$. Then because $B$ is connected and we are assuming that $B$ has common components with $\sum F_i^s$ we have statement (i). 

If there is a component   $H$ of $B$ transversal  to $f$  the assumption that $ B$ has common components with $\sum_{i=1}^3 F_i^s$  implies $B-H\neq 0$. Recall that by Lemma \ref{lem: B}, the divisor $B$ is $2$-connected and $p_a(B)=1$, so,  by Lemma \ref{B},    $H\simeq \mathbb P^1$ and $B-H$ is connected.  From Lemma \ref{components} $B-H$ is contained in fibers of $f$ and since $B-H$ is connected we obtain $B-H$  contained in $F_i^s$ for one of $i=1,2,3$, say $F_1^s$. 

In this case $HF_1^s\geq 2$ (since $B$ is 2-connected)  and $HF_i^s\geq 1$ for $i=2,3$, because  $F_i^s$ is the support of a full fiber.  Now $p_a(H+\sum_{i=1}^3 F_i^s)\geq p_a(H)+\sum_{i=1}^3 p_a(F_i^s) + H(\sum_{i=1}^3F_i^s)-3$. Since $h^0(D,\omega_D)=1$ (cf. Lemma \ref{lem: B}) $D$ cannot contain an effective divisor with $p_a\geq 2$ and so  necessarily $HF_1^s=2$ and $HF_i^s= 1$ for $i=2,3$.
\end{proof}
 \begin{lm}  \label {-10}  Let $E_j$, $j=1,\dots n$, be  the $-1$-curves contracted by   the blow down morphism  $\rho \colon X\to T$  and let  $B\preceq D$ be the divisor such that $K_X+B\sim \sum_j E_j$ (cf. Lemma \ref{lem: B}): \begin{enumerate} 
 \item if  $F_i^s\leq D-B$ then $ E_jF_i^s\leq 1$  for all $E_j$;
 \item the general fiber $F$ of $f$ satisfies $F\sum_{j=1}^nE_j=0$.
 \end{enumerate}
\end{lm} 
\begin{proof}
(i)
 Assume  that $E_jF_i^s\geq 2$. Then $h^0(E_j+F_i^s,\omega_{E_j+F_i^s})=p_a(E_j+F_i^s)\geq  p_a(F_i^s)+1\geq 1$ and so since  by the assumption    $F_i^s\leq D-B$, we have $E_j+F_i^s\leq D_1\in |K_X+D|$, contradicting Lemma \ref{ho} (cf. Lemma \ref{pa sum}).
 \medskip
 
 (ii)  Suppose otherwise, i.e., that there is  an irreducible curve $M\preceq \sum_{j=1}^n E_j$  transversal to the fibration $f$. Since $B$ and $ \sum_{j=1}^n E_j$ are disjoint (cf. Lemma \ref{lem: B}), by Lemma \ref{components}   the divisor $B$ must be contained in a fiber of $f$.

 Since each $F_i^s$ is  the support of a whole fiber we have $MF_i^s\geq 1$. 
 Since $M\preceq  \sum_{j=1}^n E_j$, we have $M^2=-l <0$. Then   among the $E_j$'s there are $E_{p_1}, ...,E_{p_{l-1}}$ such that  $E_m=M+\sum_{k=1}^{l-1} E_{p_k}$ is one of the $E_j$. This is  true by  \cite[Lemma~3.2]{ADE} if $l=1$, by  \cite[Prop.~4.1]{ADE} if $l=2$ and by \cite[Prop.~4.2]{ADE} if $l>2$.
 
  Since $E_{p_k}F_i^s\geq 0$ for all $k=1,\dots l-1$ and  $i=1,2,3$   by  Lemma \ref{ei}, we obtain $E_m(\sum F_i^s)\geq 3$.  Since $E_m(E_m+\sum _{i=1}^3F_i^s)\geq 2$,  one obtains
  \begin{itemize}
   \item[(a)] $p_a(2E_m+\sum_{i=1}^3F_i^s)\geq 2$ if  $p_g(X)=0$ and $B\neq 0$ is contained in one of the $F_i^s$. 
  \item[(b)]  $p_a(2E_m+\sum _{i=1}^3F_i^s)\geq 1$,  if $p_g(X)=1$ or  $p_g(X)=0$ and $B$ is not contained in $\sum _{i=1}^3F_i^s$; 
  \end{itemize}
 Since $E_m\leq K_X+B$, we have  $2E_m+\sum F_i^s\leq 2K_X+2B+D$,  which in case (a) yields $h^0(X, 3K_X+2B+ D)\geq 2$,  contradicting $\ol P_3(V)=1$.
 In case (b)  we have $2E_m+\sum F_i^s\leq 2(K_X+B)+ (D-B)=2K_X+B+D$,  yielding $h^0(X, 3K_X+B+ D)\geq p_g(X)+1$ and  contradicting  Lemma \ref{ho}.
  \end{proof} 
  
  \begin{cor}\label{cor: fbar} 
  The fibration $f\colon X\to \mathbb P^1$ descends to a fibration $\bar  f\colon T\to \mathbb P^1$.
  \end{cor}

 \begin{lm}  \label{genus1} The general fiber $F$ of  $\bar f$ has genus  $0$. \end{lm} 
\begin{proof}
Clearly $f$ and $\bar f$ have the same general fiber by construction.
 Since $K_T+\rho(B)=0$ and $F$ is nef, we have $FK_T=-F\rho(B)\le 0$, so either $FK_T=0$ and $F$ has genus 1, or $K_TF=-2$ and $F$ is smooth rational.
 
 Assume by contradiction that $F$  has genus $1$.  In  case $p_g(X)=0$  all the components of $B$ are contracted by $f$ since $F\rho(B)=0$, so $B$, being connected, is contained in a fiber of $f$ and we may   assume that $B$ and   $F_2^s+F_3^s$ are disjoint. 
 The divisor  $D_0=B+F_2^s+F_3^s$ satisfies   $h^0(D_0,\omega_{D_0})=p_a(B)+p_a(F_2^s)+p_a(F_3^s)=1+p_a(F_2^s)+p_a(F_3^s)$. Since  $h^0(D_0,\omega_{D_0})\le h^0(D,\omega_{D})=\ol P_1(V)=1$ (cf. Lemma \ref{lem: omegaD}), we conclude that $p_a(F_2^s)=p_a(F_3^s)=0$. If $p_g(X)=1$, applying the same argument to $D_0:=F_1^s+F_2^s+F_3^s$ we obtain $p_a(F_i^s)=0$ for $i=1,2,3$.

Denote by $\ol F_i$ the full fiber of $\bar f$ corresponding to $F_i^s$ for  $i=1,2,3$. If $p_g(X)=1$ then $T$ is minimal by Lemma \ref{lem: B}, hence $\ol F_i$ does not contain $-1$-curves for $i=1,2,3$. If $p_g(X)=0$ and $E$ is an irreducible $-1$ curve of $T$, then  $\rho(B)E=-K_TE=1$, namely $E$ meets $\rho(B)$. So in this case there is no $-1$-curve contained in  $\ol F_2+\ol F_3$. 

 By   Lemmas \ref {ei}  and  \ref{-10}, we know that     $0\leq F_i^s E_j\leq 1$  for $i=2,3$ and also for $i=1$ if $p_g(X)=1$. Therefore  $\rho^*(\rho(F_i^s))\leq F_i^s+\sum_{j=1}^n E_j$. So    the  reduced divisor $\rho(F_i^s)$ still has  $p_a=0$ for $i=1,2,3$ in the  case $p_g(X)=1$ and for $i=2,3$ in the  case  $p_g(X)=0$. 
 
 For  $i=1,2,3$ in the case $p_g(X)=1$ and for $i=2,3$ in the case  $p_g(X)=0$, the elliptic  fibers $\ol F_i$, since they   do not  contain  $-1$-curves and have support with $p_a=0$,  must be of type $^*$   (see \cite[Chp.V, \S 7]{BHPV}).  Note that, in particular, the  $\ol F_i$ cannot be multiple fibers of $\bar f$   (see \cite[Chp.V, \S 7]{BHPV}). 
 So if $\ol F$ is a fiber of type $^*$ with support   $F_0$, we have  $2F_0\geq \ol F$ if   $\ol F$ is of type $I_b^*$ ($=\tilde D_{4+b}$)   and $3F_0\geq \ol F$ if  $\ol F$ is of type $IV^*$ ($=\tilde E_6$).

On the other hand, if $p_g(X)=0$,  note  that  the sum of the Euler numbers $e(\ol F_2)+e(\ol F_3)$ cannot exceed 12, since a relatively minimal elliptic fibration 
on a rational surface has $c_2=12$.  Since  the Euler numbers of fibers of type $^*$ are always bigger than $6$ except for type $I_0^*$,  $\ol F_2$  and $\ol F_3$ must be of type $I_0^*$  and so, for instance,  $ 2(F_2^s+\sum_{j=1}^n E_j)$ contains $\rho^*(\ol F_2)=F_2$. Since  $h^0(X,F_2)=2$  we obtain $h^0(X, 2K_X+2B+2F_2^s)\geq 2$, contradicting $h^0(X, 2K_X+2D)=1$.

  Similarly, if $p_g(X)=1$  then the sum of the Euler numbers $e(\ol F_1)+e(\ol F_2)+e(\ol F_3)$ cannot exceed 24.   If one of the $\ol F_i$ is of type $I_b^*$  we have a contradiction as above and if one of the $\ol F_i$ is of type $IV^*$ ($=\tilde E_6$), then $F_i\leq 3F_i^s+3\sum_{j=1}^n E_j\leq 3K_X+3D$, contradicting $h^0(X, 3K_X+3D)=1$.  On the other hand  if all the $\ol F_i$ are of type  $II^*$ ($=\tilde E_8$)  or type $III^*$ ($=\tilde E_7$) the sum of the Euler numbers is larger than 24.
   \end{proof} 
\begin{proof}[Proof of Proposition \ref{prop: dominant0}]
By Lemma \ref{genus1} the general fiber $F$ of  $f$ has  genus $0$, and in particular we have $p_g(X)=0$. 
Since $K_XF=-2$, we have $BF=2$ and so  $B$ has a component $H$ transversal to $f$.  If $B$ has no common component with $F_i^s$, $i=1,2,3$, then  we set $D_0=B+F_1^s+F_2^s+F_3^s$ and we compute $$p_a(D_0)=p_a(B)+p_a(F_1^s)+p_a(F_2^s)+p_a(F_3^s ) +B(F_1^s+F_2^s+F_3^s)-3.$$ Since $D_0$ is reduced and connected, we have $p_a(D_0)=h^0(D_0,\omega_{D_0})\le h^0(D,\omega_D) = \ol P_1(V)=1$ (cf. Lemma \ref{pa sum} and  \ref{lem: omegaD}) and  we conclude that $p_a(F_i^s)=0$, $BF_i^s=1$ for $i=1,2,3$. On the other hand, if $B$ and $F_1^s+F_2^s+F_3^s$ have common components, then  by Corollary \ref{common} at least two of the $F_i^s$, say $F_2^s$ and $F_3^s$, have no common components with $B$ and satisfy $F_i^s B=1$ for $i=2,3$.  

So in any case $F_2^s$ contains a unique irreducible curve $\Gamma$ such that $\Gamma B\neq 0$.  Since $BF=2$ for a general fiber of $F$ and $BF_2^s=1$, $\Gamma$  appears with multiplicity $2$ in the full fiber $F_2$ containing  $F_2^s$.    The curve $\Gamma$ is not contracted by $\rho$, since $B\Gamma=1$ and $B$ does not meet the $\rho$-exceptional curves.

Write $\rho^*(\rho(\Gamma))=\Gamma +Z$, with $Z$ an exceptional divisor. Since $B=  \rho^*(\rho(B))$ the projection formula gives 
$$1=B\Gamma= \rho^*(\rho(B))\Gamma=\rho^*(\rho(B))(\Gamma+Z)=\rho^*(\rho(B))\rho^*(\Gamma)=\rho(B)\rho(\Gamma).$$
 Since $K_T+\rho(B)=0$, we have $K_T\rho(\Gamma)=-1$. The curve $\rho(\Gamma)$ is contained in a fiber of $\bar f$, so $\rho(\Gamma)^2\le 0$ and $\rho(\Gamma)$ is a $-1$-curve  by the adjunction formula.   
So $\ol F_2=2\rho(\Gamma) +C$, where $C$ does not contain $\rho(\Gamma)$. The components of $C$ do not meet $\rho(B)=-K_T$, hence they are all $-2$-curves. 
From $\rho(\Gamma) \ol F_2=0$, we get $C\rho(\Gamma)=2$. If there are two distinct components $N_1$ and $N_2$ of $C$ with  $\rho(\Gamma)N_i=1$, then $N_1$ and $N_2$ are disjoint, since the dual graph of $\ol F_2$ is a tree. So $(2\rho(\Gamma)+N_1+N_2)^2=0$ and therefore $\ol F_2=2\rho(\Gamma)+N_1+N_2$. If there is only a component $N$ of $G$ with $N\rho(\Gamma)=1$, then  $N $ appears in $\ol F_2$ with multiplicity 2 and $M_1:=\rho(\Gamma)+N$ is a (reducible and reduced ) $-1$-curve such that $\ol F_2=2M_1+C_1$ and $C_1$ and $M_1$ have no common component. So we may repeat the previous argument and either write $\ol F_2=2M_1+N_1+N_2$, where $N_1,N_2$ are disjoint  $-2$-curves contained in $C_1$ with $N_i M_1=1$, or $\ol F_2=2(M_1+N)+C_2$, where   $N$ is a component of $C$  such that $M_2:=M_1+N$ is a $-1$-curve and $C_2$ and $M_2$ have no common component. This process must of course terminate, showing that  $\ol F_2=2M_0+N_1+N_2$, where $M_0$ is a reduced $-1$-curve, and $N_1$ and $N_2$ are disjoint $-2$-curves not contained in $M_0$. In particular we have shown that $\ol F_2$ has no component of multiplicity $>2$.

Since  (as in the proof of Lemma \ref{genus1}) $\rho^*(\rho (F_2^s))\leq F_2^s+\sum_{j=1}^n E_j$ we conclude that $2F_2^s+2\sum_{j=1}^n E_j$ contains the full fiber $F_2$ of $f$ 
and as in the in the proof of Lemma \ref{genus1}   we obtain $h^0(X, 2K_X+2B+2F_2^s)\geq 2$, contradicting $h^0(X, 2K_X+2D)=\ol P_2(V)=1$.

So we have excluded all the possibilities for the genus of the general fiber $F$ of the fibration induced by $a_V$ if $a_V$ is not dominant and the proof is complete.
\end{proof}

\begin{ex}\label{rem: sharp}

The hypothesis $\ol P_3(V)=1$ in Proposition \ref{prop: dominant0} is  necessary. A K3 surface with an elliptic fibration with three singular fibers of type $IV^*$(cf. the proof of Lemma \ref{genus1}) does exist. Take for instance  the elliptic curve $C$ with an automorphism $h$ of order 3,  let $\mathbb Z/3$ act on $C \times C$  by $(x,y)\mapsto (hx,h^2y)$ and denote by  $X_0$ the quotient surface. Then $X_0$ has 9 singular points of type $A_2$, and the first  projection $C\times C\to C$ descends to an isotrivial elliptic fibration $X_0\to C/\mathbb Z_3\cong \mathbb P^1$ with three ``triple'' fibers, each containing three of the 9 singular points. The minimal resolution $X$ of $X_0$ is a K3 surface with an isotrivial elliptic fibration with three fibers of type $IV^*$.
Alternatively, $X$ can be constructed as  the minimal resolution of a simple   $\mathbb Z_3$-cover of 
$\mathbb P^1\times \mathbb P^1$ branched   over three fibers of one of the fibrations plus three fibers of the other. 

Consider the surface $V:=X\backslash\{F_1^s,\;F_2^s,\;F_3^s\}$, where the $F_i^s$ denote the supports of the 3 fibers of type IV*. Then $V$ has $\ol P_1(V)=\ol P_2(V)=1$ and $\ol q(V)=2$ (see  Proposition \ref{prop: form}), and  its quasi-Albanese map is the restriction of the elliptic fibration $X\to\mathbb P^1$ and so it is not dominant. 
\end{ex}

\section{Proof of Theorem A} \label{sec: main}
Thanks to what we have proven in the previous section, we know that in the assumptions of Theorem \ref{Main}  the quasi-Albanese map of $V$ is dominant. Since $V$ and $\operatorname{Alb}(V)$ have the same dimension, we can conclude that $a_V$ is generically finite. In particular there is a generically finite morphism $g:X\rightarrow Z$, where $Z$ is the compactification of $\operatorname{Alb}(V)$ described in Proposition \ref{prop: compactification}. Now, our proof boils down to  the following two facts
\begin{enumerate}
    \item[1.] the morphism $g$ has  degree 1;
    \item[2.] all the components of $D$ that are not mapped  to the boundary  $\Delta$ of $\operatorname{Alb}(V)$ are contracted by $g$.
\end{enumerate}
When $q(V)=2$, the first assertion is Corollary  \ref{thm:qx=2}.  The second statement can be proven  by contradiction (see \S \ref{ssec: conclusion}).\par
The situation is  more involved when $q(V)<2$.  We start by proving  a slightly weaker version of  assertion 2 (Lemma \ref{lem: C}) and we use it  to  show that the finite part of the Stein factorization of $g$ is \'etale over $A(V)$ (in case $q(V)=1$ this requires also a topological argument). Now the universal property of the quasi-Albanese map implies that $g$ has indeed  degree 1. Finally,  we complete the proof of assertion 2 by means of a local computation.  


\subsection{Preliminary steps}
\begin{nota}
We let $V$ be a smooth open algebraic surface $V$ with $\overline{q}(V)=2$. 
We assume $\ol P_1(V)=\ol P_2(V)=1$ if $q(V)\ge 1$ and $\ol P_1(V)=\ol P_3(V)=1$ if $q(V)=0$.

If $q(V)\le 1$ fix a compactification  $Z$  with boundary $\Delta$ of the quasi-Albanese variety $A(V)$ of $V$ as follows:
\begin{itemize}
\item if $q(V)=1$, we take $Z$  a $\mathbb P^1$-bundle over the compact part $A$ of $A(V)$ as in Corollary \ref{cor: Z}, and we write $\Delta=\Delta_1+\Delta_2$, where $\Delta_i$ are disjoint sections of $Z\to A$;
\item if $q(V)=0$ (and thus $A(V)=\mathbb G_m^2$) we take $Z=\mathbb P^1\times \mathbb P^1$ with the obvious choice of the boundary $\Delta$. 
\end{itemize}

We fix a compactification $X$ of $V$ with snc boundary $D$ such that the quasi-Albanese map $a_V\colon V\to A(V)$ extends to a morphism $g\colon X \to Z$. In addition,  we write $H:=g\inv (\Delta)$ (set-theoretic inverse image) and  in case $q(V)=1$ we also write $H_i=g\inv (\Delta_i)$, $i=1,2$. Note that $D\succeq H$ by construction. When $q(V)=1$ we denote by $a_X\colon X\to A=A(X)$ the Albanese map of $X$.
\end{nota}

The proof is quite involved so we break it into several smaller steps and we examine the case $q(V)\leq 1$ since by Corollary  \ref{thm:qx=2}, for $q(V)=2$ we already know that $a_V$ is birational.

\begin{lm}\label{lem: pg0}
If $q(V)\le 1$, then the divisor  of poles of a  generator of $H^0(X, K_X+D)$ is a non zero subdivisor of $H$.\newline 
In particular, $p_g(X)=0$ and $h^0(X, K_X+H) =1$.
\end{lm}
\begin{proof} The vector space $H^0(Z,\Omega^1_Z(\log \Delta))$ is generated by two logarithmic $1$-forms $\tau_1$ and $\tau_2$ such that $\omega:=\tau_1\wedge \tau_2$  vanishes nowhere on $A(V)$ (cf. Proposition \ref{prop: compactification}) and has poles exactly on $\Delta$. More precisely, if $q(V)=1$ then we can take $\tau_1$ to be the pull back of a nonzero regular 1-form on $A=A(X)$ via the projection $Z\to A$ and $\tau_2$ a logarithmic $1$-form with poles on $\Delta_1$ and $\Delta_2$ (cf. proof of Corollary \ref{cor: Z}); if $q(V)=0$ then we can take $\tau_1$ and $\tau_2$ to be  pullbacks of  non zero logarithmic forms via the two projections $ \mathbb G_m^2\to \mathbb G_m$. 

 Since $g$ is surjective by Propositions \ref{prop: dominant1} and \ref{prop: dominant0}, $g^*\omega$ is a non-zero logarithmic form on $X$ and a local computation shows that if $\Gamma$ is an irreducible component of $H$ not contracted by $g$, then $g^*\omega$ has a pole along $\Gamma$. On the other hand the poles of $\omega$ are contained in $H$ by construction. 

Since $\ol P_1(V)=h^0(X, K_X+D)=1$, it follows immediately that $p_g(X)=0$  and $h^0(X, K_X+H) =1$.
\end{proof}

\begin{lm}\label{lem: connected}
\begin{enumerate}
\item If $q(V)=0$, then $H$ is connected and $h^0(\omega_H)=1$;
\item If $q(V)=1$, then for $i=1,2$ the divisor $H_i$ is connected and $h^0(\omega_{H_i})=1$ 
\end{enumerate}
\end{lm}
\begin{proof}
(i) The divisor $H$ is connected since it is the support of  the nef and big divisor $g^*\Delta$. Lemma \ref{lem: omegaD} gives $h^0(H,\omega_H)=1$ since $p_g(X)=0$ by Lemma \ref{lem: pg0}. 
\smallskip

(ii) 
Consider  the exact sequence 
$$H^1(X, -H)\to H^1(X, \OO_X)\to H^1(H, \OO_H).$$
The second map in the sequence is non zero, since $H$ is mapped onto $A$ by the Albanese map $a_X$, so in this case Lemma \ref{lem: omegaD} gives  $h^0(H, \omega_H)=2$.
The divisor $H$ is the disjoint union of $H_1$ and $H_2$, so  $2= h^0(H, \omega_H)=h^0(H_1,\omega_{H_1})+h^0(H_2,\omega_{H_2})$. For $i=1,2$ let $S_i$ be a component of $H_i$ such that $g(S_i)=\Delta_i$. Since $\Delta_i$ has geometric genus 1, by Lemma \ref{pa sum} we have $1\le h^0(S_i,\omega_{S_i})\le h^0(H_i,\omega_{H_i})$.  So we have  $h^0(S_i,\omega_{S_i})= h^0(H_i,\omega_{H_i})=1$ for $i=1,2$ and $S_i$ is the only component of $H$ with $p_a>0$.
Assume now by contradiction that $H_i=B_i+C_i$, where $B_i$ and $C_i$ are disjoint non zero effective divisors and $S_i\preceq B_i$. Then all the components  of $C_i$  are rational and so their images via $g$ are contained in fibers of $Z\to A$. Since  $g(C_i)\subset \Delta_i$, it follows that $g(C_i)$ is a finite set. So the intersection form on the set of components of $C_i$ is negative definite, contradicting the fact that $B_i+C_i$ is the support of the nef divisor $g^*\Delta_i$. 
\end{proof}
 Since $K_Z+\Delta=0$, the logarithmic ramification formula \eqref{eq: logram} gives $K_X+D\sim \ol R_g$. We aim  to show that the components of $\ol R_g$ not contained in $H$ are contracted to points. 
 We begin with a simple observation:
\begin{lm} \label{lem: Gamma}
If\,  $\Gamma$ is  an irreducible component of $\ol R_g$, then 
$h^0(X, K_X+H+\Gamma)\le 1$.
\end{lm} 
\begin{proof} 
We have $(K_X+H)+\Gamma\preceq (K_X+D)+\ol R_g=2(K_X+D)$. 
So $h^0(X, K_X+H+\Gamma)\le \ol P_2(V)=1$.
\end{proof}
\begin{lm}\label{lem: C}
Let $C$ be the union of all the components of $ \ol R_g$ that are not contained in $H$ and  are not contracted by $g$. Then:
\begin{enumerate}
\item if $q(V)=0$, then  $C=0$;
\item if $q(V)=1$ and $C>0$, then $C$ is irreducible and $g(C)$ is a ruling of $Z$.
\end{enumerate}
\end{lm}
\begin{proof}
(i) Let $\Gamma$ be an irreducible  component of $C$. Since $g(\Gamma)$ is a curve not contained in $\Delta$ and since $\Delta$ is the union of two fibers of the first projection $\mathbb G^2_m\to \mathbb G_m$ and two fibers of the second projection, $g(\Gamma)\cap \Delta$ contains at least two distinct points. So $\Gamma\cap H$ also contains at least two distinct points, and therefore $H\Gamma\ge 2$ and
$p_a(H+\Gamma)= p_a(H)+p_a(\Gamma)+H\Gamma-1\ge p_a(H)+1$. Since both the reduced  divisors $H$ and $H+\Gamma$ are connected by Lemma \ref{lem: connected}, we have $p_a(H)=h^0(H,\omega_H)$ and
 $h^0(H+\Gamma, \omega_{H+\Gamma})=p_a(H+\Gamma)\ge h^0(H,\omega_H) +1=2$,
 where the last equality follows from Lemma \ref{lem: connected}. 
Now  Lemma \ref{lem: omegaD} gives  $h^0(X, K_X+H+\Gamma)= h^0(H+\Gamma, \omega_{H+\Gamma})\ge 2$ (recall that $p_g(X)=0$ by Lemma \ref{lem: pg0}), contradicting Lemma \ref{lem: Gamma}.
\smallskip

\noindent (ii) Let again $\Gamma$ be a component of $C$ such that  $g(\Gamma)$ is not a ruling of $Z$. Then $\Gamma$ dominates $A$ and therefore $p_a(\Gamma)\ge 1$.
If $\Gamma$ is disjoint from $H$, then $h^0(H+\Gamma, \omega_{H+\Gamma})=h^0(H,\omega_H)+h^0(\Gamma,\omega_{\Gamma})\ge 3$ by Lemma \ref{lem: connected}.
If $\Gamma$ intersects $H$, then it intersects both $H_1$ and $H_2$, since they support the numerically equivalent divisors $g^*\Delta_1$ and $g^*\Delta_2$. So $\Gamma+H$ is connected and $h^0(\Gamma +H,\omega_{\Gamma+H})=p_a(\Gamma+H)\ge p_a(\Gamma)+p_a(H)+1=3$. In either case, Lemma \ref{lem: omegaD} gives $h^0(X, K_X+H+\Gamma)\ge 2$, contradicting Lemma \ref{lem: Gamma}. So we conclude that $g(\Gamma)$ is a ruling of $Z$. 
Assume that $C$ has at least two components $\Gamma_1$ and $\Gamma_2$: then the same argument as above gives $h^0(X, K_X+H+\Gamma_1+\Gamma_2)\ge 2$, contradicting Lemma \ref{lem: Gamma} again.
\end{proof}

Now we consider the Stein factorization $X\overset{\nu}{\to} \ol X\overset{\ol g}{\to}Z$ of $g$.
\begin{lm} \label{lem: etale}
The morphism $\ol g$ is \'etale over $A(V)$.
\end{lm}
\begin{proof}
Let $m:=\deg g$; if $m=1$ then $\ol g$ is an isomorphism and the claim is of course true. So we may assume  $m>1$. 

The map $\ol g$ is finite by construction, so by purity of the branch locus it is enough to show that there is no component of the (usual) ramification divisor  of $g$ that is not contained in $H$ and is not contracted to a point. By Lemma \ref{lem: Rg} such a curve is a component of the logarithmic ramification divisor $\ol R_g$. So if $q(V)=0$ the statement follows directly from Lemma \ref{lem: C}. Therefore we assume for the rest of the proof that $q(V)=1$.

Again by Lemma \ref{lem: C}, if $\ol g$ is not \'etale then there is exactly one irreducible curve $\Gamma$ in $R_g$ such that $\Gamma$ is not contained in $H$ and is not contracted by $g$, and the image of $\Gamma$ is a ruling $\Phi$ of $Z$. So $\ol g$ restricts to a connected  \'etale cover $q\colon \wt W \to W:=A(V)\setminus \Phi$. 
The preimage in $X$ of a ruling of $Z$ is a fiber of the Albanese map $a_X\colon X\to A(V)$ and so it is connected. So $q$ restricts to a connected cover of $\mathbb G_m$.

Since algebraically trivial line bundles are topologically trivial, Corollary \ref{cor: Z} implies that $Z$ is homeomorphic to $A\times \mathbb P^1$, $A(V)$ is homeomorphic to $A\times \mathbb G_m$ and $W$ is homeomorphic to $(A\setminus\{a\})\times \mathbb P^1$, where $a\in A$ is a point. Fix base points in $\wt W$ and $W$ and denote by $N$ the subgroup of index $m$ of $\pi_1(W)\simeq \pi_1(A\setminus\{a\})\times \pi_1( \mathbb G_m)$ corresponding to the cover $q$.  If $\gamma$ is a generator of $\pi_1(\mathbb G_m)$, we have $N\cap \pi_1(\mathbb G_m)=<\gamma^m>$. So the elements $1, \gamma,\dots \gamma^{m-1}$ represent distinct left cosets  of  $\pi_1(W)$ modulo $N$. Since  $\pi_1(\mathbb G_m)$ is a central subgroup of $\pi_1(W)$, it follows that left and right cosets modulo $N$ coincide, namely $N$ is a normal subgroup of $\pi_1(W)$ and the quotient $\pi_1(W)/N$ is cyclic of order $m$. Since the map $\pi_1(W)\to \pi_1(A(V))$ is just abelianization, $N$ is the preimage of an index $m$ subgroup $\ol N<\pi_1(A(V))$. This shows that $q$ extends to an \'etale cover $q'\colon W'\to A(V)$. Since by \cite[Thm.~3.4]{DG94} both $q$ and $q'$ extend uniquely to an analytically branched covering of $Z$, it follows that $\ol g\colon \ol X\to Z$ extends $q'$, that is, $\ol g$ is \'etale over $A(V)$ as claimed. 
\end{proof}
\subsection{Conclusion}\label{ssec: conclusion}
We are finally ready to complete the proof of Theorem A.
\begin{proof}[Proof of Theorem A]

Consider the case $q(V)=2$ first. By Corollary \ref{thm:qx=2},  we only need to show that all components of $D$ are contracted by $a_X$. So assume for contradiction that there is an irreducible component $\Gamma$ of $D$ that is not contracted by $a_X$ and denote by $\ol \Gamma$ its image in $A(X)$; note that the map $ \Gamma\to \ol \Gamma$ is birational. The geometric genus  of $ \ol\Gamma$ is  positive  and if  it is equal to 1, then $\ol \Gamma$ is a translate of an abelian subvariety of $A(X)$, so in particular it is smooth. 

Assume first that $p_a(\Gamma)=1$: then $\Gamma$ is smooth of genus 1 and $\Gamma \to \ol \Gamma$ is an isomorphism. It follows that $a_X^*\ol \Gamma \preceq K_X+\Gamma\preceq K_X+D$. Since $h^0(A(X), 2\ol \Gamma)=2$, we have a contradiction to $\ol P_2(V)=2$.  If $p_a(\Gamma)\ge 2$, then by Serre duality and Riemann--Roch for all $\alpha\in \Pic^0(X)$  we have $h^0(X, K_X+\Gamma+\alpha)\ge \chi(K_X+\Gamma)=p_a(\Gamma)-1\ge 1$. Lemma \ref{lem: -Z} gives $h^0(X, 2(K_X+\Gamma))\ge 2$, contradicting again the assumption $\ol P_2(V)=1$. So all the components of $D$ are $a_X$-exceptional.  
\smallskip

From now  we assume $q(V)\le 1$.
By Lemma \ref{lem: etale}, the quasi-Albanese map $a_V\colon V\to A(V)$ factors through an \'etale cover of degree $m:=\deg g$. Since a finite  \'etale cover of a quasi-abelian variety is also a quasi-abelian variety (Proposition \ref{prop: quasiab}) by the universal property of $a_V$ (cf. Theorem \ref{thm: quasiA}) we have $m=1$, namely $g$ is birational. To complete the proof we need to show that all the components of $D-H$  are contracted by $g$. Since  $D-H\preceq \ol R_g$ by Lemma \ref{lem: Rg}, in case $q(V)=0$ the claim follows by Lemma \ref{lem: C}.

Assume $q(V)=1$. By Lemma \ref{lem: C} there is at most an irreducible curve $\Gamma \preceq D-H$ such that $\Gamma$ is not contracted by $g$, and $g(\Gamma)$ is a ruling $\Phi$ of $Z$. We are going to show that $g^*\Phi\preceq \ol R_g$, hence $g^*(2\Phi)\preceq 2\ol R_g\sim 2(K_X+D)$. Since $h^0(Z, 2\Phi)=2$, this contradicts  the assumption $\ol P_2(V)=1$.

Write $g^*\Phi=\Gamma + \sum_i\alpha_iC_i$, where the $C_i$ are distinct irreducible curves contracted by $g$ and $\alpha_i \in \mathbb N_{>0}$. Let $u,v$ be local coordinates on $Z$ centered at the point $P_i:=g(C_i)$ such that $u=0$ is the ruling $\Phi$ of $Z$; if in addition $P_i\in \Delta$, we assume that $v$ is a local equation of $\Delta$. At a general point of $C_i$ we have $u=x^{\alpha_i}a$, $v=x^{\beta_i}b$, where $x$ is a local equation for $C_i$, $a,b$ are non zero regular functions and $\beta_i>0$ is an integer.  A simple computation gives
\begin{equation}\label{eq: forms}g^*\frac{dv}{v}=\beta_i\frac{dx}{x}+\frac{db}{b}, \qquad   g^*\left(\frac{du}{u}\wedge \frac{dv}{v}\right)=\frac{dx}{x}\wedge \left(\alpha_i\frac{db}{b}-\beta_i\frac{da}{a}\right).
\end{equation}
Note that $\alpha_i\frac{db}{b}-\beta_i\frac{da}{a}$ is a regular 1-form. 
So if $P_i\in \Delta$, then $C_i\preceq D$,  $K_Z+\Delta$ is locally generated by $u\wedge \frac{dv}{v}$, $K_X+D$ is locally generated by $\frac{dx}{x}\wedge dy$ and $C_i$ appears in $\ol R_g$ with multiplicity $\ge \alpha_i$.

If $P_i\notin \Delta$, then $K_Z+\Delta$ is locally generated by $du\wedge dv$, $K_X+D$ is locally generated by $\frac{dx}{x}\wedge dy$ or by $dx\wedge dy$ according to whether $C_i\preceq D$ or not; in either case $C_i$ appears in $\ol R_g$ with multiplicity $\ge \alpha_i+\beta_i-1\ge \alpha_i$. Summing up, we have shown $g^*\Phi\preceq \ol R_g$, as promised.
\end{proof}
\bibliography{QuasiAbelian.bib}

\end{document}